\documentclass[11pt]{article}
\usepackage[utf8]{inputenc} 
\usepackage{amsmath,amsthm} 
\usepackage{amssymb} 
\usepackage{color}
\usepackage{breqn}
\usepackage{color}
\usepackage{setspace}
\usepackage{vmargin}
\usepackage{xfrac}
\usepackage{pgf}
\usepackage{tikz}
\usetikzlibrary{arrows,automata}
\usepackage[mathscr]{euscript}
\usepackage[neverdecrease]{paralist}
\usepackage{ragged2e}

\setpapersize{A4}
\setmargins{3 cm}        
{1.5cm}                   
{15cm}                  
{23.42cm}                 
{12pt}                    
{1cm}                     
{0pt}                     
{2cm}                     

\usepackage{enumitem}

\usepackage{subfig}
\usepackage{wrapfig}
\usepackage{calc}
\newtheorem{theorem}{Theorem}

\theoremstyle{definition}
\newtheorem{definition}{Definition}

\newtheorem{proposition}{Proposition}

\newtheorem{observation}{Observation}

\newtheorem{lemma}{Lemma}

\DeclareMathOperator{\conv}{conv}
\newcommand{\Tv}{\mathit{Tv}}
\newcommand{\Tn}{\mathit{Tn}}

\begin{document}
\title{From word-representable graphs to altered Tverberg-type theorems}

\author{D.\ Oliveros and A.J.\ Torres }

\maketitle
\begin{abstract}
Tverberg's theorem says that a set with sufficiently many points in
$\mathbb{R}^d$ can always be partitioned into $m$ parts so that the $(m-1)$-simplex is the
(nerve) intersection pattern of the convex hulls of the parts. In \cite{TverbergTipeTheorems} the authors investigate how other simplicial complexes arise as nerve complexes once we have a set with sufficiently many points. In this paper we relate the theory of word-representable graphs \cite{WordsandGraphsBook} as a way of codifying $1$-skeletons of simplicial complexes to generate nerves. In particular, we show that every $2$-word-representable triangle-free graph, every circle graph, every outerplanar graph, and every bipartite graph could be induced as a nerve complex once we have a set with sufficiently many points in $\mathbb{R}^d$ for some $d$.  
\end{abstract}


\section{Introduction}

One of the most celebrated  and beautiful theorems in discrete geometry is due to H.\ Tverberg and states the following:


\begin{theorem}[Tverberg Theorem 1966 \cite
{tverberg1966generalization}]
Every set $S$ with at least $(d + 1)(m - 1) + 1$ points in Euclidean $d$-space $\mathbb{R}^d$ can be partitioned into $m$ parts $P = S_1 ,\dots , S_m$ such that all the convex hulls of these parts have nonempty intersections. The special case of a bi-partition $m=2$ is known as Radon's Lemma. 
\end{theorem}
 
 Given a family $F = \{F_1 , \dots  , F_m \}$ of convex sets in $\mathbb{R}^d$, the nerve $\mathcal{N}(F)$ of $F$ is the simplicial complex with vertex set $[m]:= \{1,2,...,m \}$ whose faces are $I \subset [m]$ such that $\cap_{i\in I} F_i \neq \emptyset$ (see \cite{matousek2013lectures},\cite{tancer2013intersection} for details).

Given a finite collection of points $S\subset \mathbb{R}^d$ and an $m$-partition into $m$ color classes $P=S_1 , \dots  , S_m$ of $S$, the nerve of the partition, $\mathcal{N}(P)$ is the nerve complex $\mathcal{N}(\{\conv(S_1 ),\dots,$ $\conv(S_m )\})$, where $\conv(S_i)$ is the convex hull of the elements in the color class $i$. 

We will say that a simplicial complex $\mathcal{K}$ is \emph{partition-induced} on a finite set of points $S\subset\mathbb{R}^d$ if there exists a partition $P$ of $S$ into color classes such that the nerve of the partition is isomorphic to $\mathcal{K}$.  Then  $\mathcal{K}$ is \emph{$d$-Tverberg} if there exists a constant $\Tv(\mathcal{K}, d)$ such that for every point set $S$ with at least $\Tv(\mathcal{K}, d)$ points, there exists a partition $P$ such that $\mathcal{N}(P)=\mathcal{K}$. The minimal such constant $\Tv(\mathcal{K}, d)$ is called the \emph{Tverberg number} for $\mathcal{K}$ in dimension $d$.

We can think of Tverberg's theorem as a Ramsey-type  theorem,  where one studies how every sufficiently large system (set of points) must contain a large well-organized sub-system.  ``Sufficiently large"  means that for every set of points $S$ with at least $\Tv(\mathcal{K}, d)=(d+1)(m-1)+1$ points, there always exists a partition $P$ into $m$ color classes of $S$ such that $\mathcal{N}(P)=\mathcal{K}$, where $\mathcal{K}$ is  a simplex.  This allows us to rephrase  Tverberg's theorem as follows:

\begin{theorem}[Tverberg's theorem rephrased]
The $(m-1)$-simplex is a $d$-Tverberg complex for all $d \geq 1$, with Tverberg number $\Tv(\mathcal{K}, d)=(d+1)(m-1)+1$.
\end{theorem}

In \cite{TverbergTipeTheorems}, the authors investigate other possible nerves and point out that Tverberg's theorem can be seen as a special case of a more general situation, showing some new Ramsey-Tverberg-type results. In particular, nerves such as trees and cycles are $d$-Tverberg for every $d$. In this paper we follow the same philosophy and find an interesting relationship between this problem and word-representable graphs. 

One of the most recent findings in the area of graph theory is the notion of word-representable graphs, which is a common generalization of several well-studied classes of graphs such as circle graphs, comparability graphs, $3$-colorable graphs and graphs of degree at most $3$ (also known as subcubic graphs). A graph $G = (V, E)$ is \emph{word-representable} if there exists a word $W$ over the alphabet $V$ such that letters $x$ and $y$ alternate in $W$ if and only if $\{x, y\}\in E$ for each $x\not= y$. 
 In particular, if the length of the alternation is $k$, the word is called $k$-\emph{representable}. (See Definitions \ref{defwordrepresentable} and  \ref{defkwordrepresentable} and  the book  \cite{WordsandGraphsBook} by S.\ Kitaev and V.\ Lozin,  for the state of art of this topic.) Not all graphs are word-representable, but we have observed that with a slight generalization and allowing  the word $W$ to contain a $k$-alternating subword (see Definition \ref{defgeneralwordrepresentable}), every graph is ``\emph{general} $k$-\emph{word-representable}" for some $k$; see Definition \ref{General d-words-representable graphs}. Now determining the minimum $k$ seems an interesting problem on its own. 

In this paper we observe an engaging connection between general $k$-word-representable graphs and Ramsey-Tverberg-type results, where nerve structures are shown to arise once we have sufficiently many points. 

The main philosophy for the proofs of our results relies on the following four steps.

\begin{itemize}

\item[{Step} 1)] Suppose we would like to show that a certain simplicial complex $\mathcal{K}$ with vertex set $[m]$  is $d$-Tverberg for some $d$. Assume that the $1$-skeleton of $\mathcal{K}$ is a graph that is general $r$-word-representable by some  word $W$ with  $m$ letters.   Then use $d=r-2$.

\item [{Step} 2)] Apply the famous Ramsey-type Erd\H{o}s-Szekeres result in dimension $d=2$ \cite{Erdos1987} or the multi-dimensional version of the Erd\H{o}s-Szekeres theorem (due to Gr\"{u}nbaum \cite{Gbook} and Cordovil and Duchet \cite{CordovilDuchet};  see also \cite[Chapter 9]{OMbook}, and the survey \cite{Morris}) in dimension $d\geq 2$ with sufficiently many points to ensure we have at least $m$ points  on  some curve that is combinatorially equivalent to  moment curve. 

\item[{Step} 3)] Label these $m$-vertices on this ``moment curve" according to the letters of the word $W$. 
This naturally yields a partition. Using the Radon partitions for cyclic polytopes given by M.\ Breen in \cite{breen1973primitive} we obtain the desired nerve $\mathcal{K}$ for points in convex position in $\mathbb{R}^d$. 

\item[{Step} 4)] Extend the partition to the rest of the points in $\mathbb{R}^d$ without altering the desired nerve $\mathcal{K}$.

\end{itemize}

These tools are enough to show the existence of a Tverberg number $\Tv(\Tn, d)$, but the bounds are far from tight. 

Although it is probably true, it seems a difficult problem at this point to show whether all four steps work for all simplicial complexes. In this paper we are able to prove, among other things, the following  theorems.

\begin{theorem}\label{mainTheo1}
Every general $2$-word-representable triangle-free graph $G=(V,E)$  is $2$-Tverberg.
\end{theorem}

\begin{theorem}\label{maintheo2circlegraph}
Every circle graph is $2$-Tverberg.
\end{theorem}

These two theorems imply that graphs such as outerplanar graphs are $2$-Tverberg, and prove that trees and cycles are $2$-Tverberg. 
 
\begin{theorem}\label{sin_triangulos_encaje_Md}
Every triangle free graph $G=(V,E)$ is partition induced on a sufficiently large set of points in convex position for every integer $d \geq d_0$ for some $d_0\geq 1$.
\end{theorem}

\begin{theorem}\label{maintheo3bipartition}
Every bipartite graph is $d$-Tverberg for some $d$. 
\end{theorem}


In Section \ref{General d-words-representable graphs} we review the concept of word-representable and $k$-word-representable graphs, and give  a generalization of these concepts which helps us to codify every given  triangle-free graph $G$. We further give examples of graphs that are planar, bipartite and not general $2$-word-representable graphs, thus not $2$-Tverberg.

In Section \ref{nerves over cyclic polytopes} we show that every triangle-free general $d$-word-representable graph is partition-induced on a set of points in convex position in dimension $d$. In Section \ref{Tverberg type theorems and extensions} we  show that
if a simplicial complex is partition-induced on a set of points in convex position in the plane, then it is  $2$-Tverberg (Theorem \ref{2-extension}), which proves in particular Theorems \ref{mainTheo1} and \ref{maintheo2circlegraph}. Finally, in the last part of this section we will prove Theorem \ref{maintheo3bipartition} using an specific new codification in $d$-word-representable graphs and  Radon partitions on cyclic polytopes.

\section{Word-representable graphs} \label{General d-words-representable graphs}

\subsection{Word-representable and \emph{k}-word-representable graphs}

A \emph{word} $W$ in an alphabet $A(W)$ is simply a finite sequence of letters $a_1a_2\dots a_r$ with $a_i\in A(W)$.  A \emph{sub-word}  $W^{\prime}$ of $W$ is a subsequence of $W$, and a word $\overline{W}$ is a \emph{factor} of $W$ if $W = W_1\overline{W}W_2$ for possibly empty words $W_1$ and $W_2$. We denote the letter of $W$ in position $i$ by $W(i)$ for $1\leq i \leq |W|$. 

The concepts of word-representable graphs  and $k$-word-representable graphs were introduced by Kitaev in \cite{Kitaev} and are defined as follows.

\begin{definition}\label{defwordrepresentable}
A graph $G=(V,E)$ is \emph{word-representable} if there exists a word $W$ over the alphabet $V$ such that two letters $x$ and $y$ alternate in $W$ if and only if $\{x,y\} \in E$. We say that $W$ represents $G$, and $W$ is called a word-representant for $G$.
\end{definition} 


\begin{definition}\label{defkwordrepresentable}
A graph $G$ is \emph{$k$-word-representable} if there exists a word $W$ such that $W$ is a word-representant for $G$ and every letter appears in $W$ exactly $k$ times. The minimal $k$ such that a graph is $k$-word-representable is called the graph's representation number, and it is denoted by $\mathcal{R}(G)$. Also, $\mathcal{R}_k = \{G | \mathcal{R}(G) = k\}$.
\end{definition}


There are many interesting results and contributions in word-representable and $k$-word-representable graphs (see \cite{WordsandGraphsBook} for an excellent treatment of the theory). Below we will list some of them that we will use in this work.

\begin{lemma}[\cite{kitaev2008representable}]\label{l1}
A graph $G$ is word-representable if and only if it is $k$-word-representable for some $k$.
\end{lemma}

\begin{lemma}[\cite{kitaev2008representable}]\label{l2}
A $k$-word-representable graph $G$ is also $(k + 1)$-word-representable.
\end{lemma}



\begin{lemma}[\cite{halldorsson2011alternation}]\label{l4}
For a graph $G$ different from a complete graph, $G$ is $2$-word-representable if and only if $G$ is a circle graph (intersection graph of a set of chords with vertices on a circle).
\end{lemma}

\subsection{General \emph{d}-word-representable graphs}
We now introduce a very similar concept that results in a generalization of $k$-word-representable graphs. 

\begin{definition}\label{defgeneralwordrepresentable}
Given a word  $W$ in an alphabet $A(W)$ for $|A(W)|\geq2$, we say that two letters $a,b\in A(W)$ are \emph{$d$-intersecting} for $d\geq 1$  if there exists an  alternating sub-word of $W$  of length at least $d+2$.    
\end{definition}

\begin{definition}
A graph $G=(V,E)$ is general $d$-word-representable if there exists a word $W$ from alphabet $V$ such that for every edge $\{x,y\}\in E$, letters $x$ and $y$ are $d$-intersecting in $W$. We say that $W$ general $d$-word-represents the graph $G$, and $W$ is called a general $d$-word-representant for $G$. The minimal $d$ such that a graph is general $d$-word-representable is called the graph's general $d$-word-representation number, and it is denoted by $\mathcal{GR}(G)$. 
\end{definition}

Clearly, this last definition is more general in the sense that every word-representable graph $G$ is also general $d$-word-representable for some $d$. In fact, using the same proof that the authors gave for the lemma \ref{l2} in \cite{kitaev2008representable} the following result is true. 

\begin{proposition}\label{d is d+1 }
Every general $d$-word-representable graph $G$ is also general $(d + 1)$-word-representable.
\end{proposition}

Now consider, for example, the wheel $W_5$ (see Figure \ref{wheel5}). It is known \cite{kitaev2008representable} that this graph is not $k$-word-representable; however, the word $W=156216326436546$  general $4$-word-represents this graph. Note for instance that for $1$ and $6$, the sub-word $1616666$ is not an alternating sequence, but contains one of length $4$ that is $1616$.  
The following proposition observes that every graph is a general $d$-representable graph.
\begin{figure}[h!]
\begin{center}
\includegraphics[scale=.7]{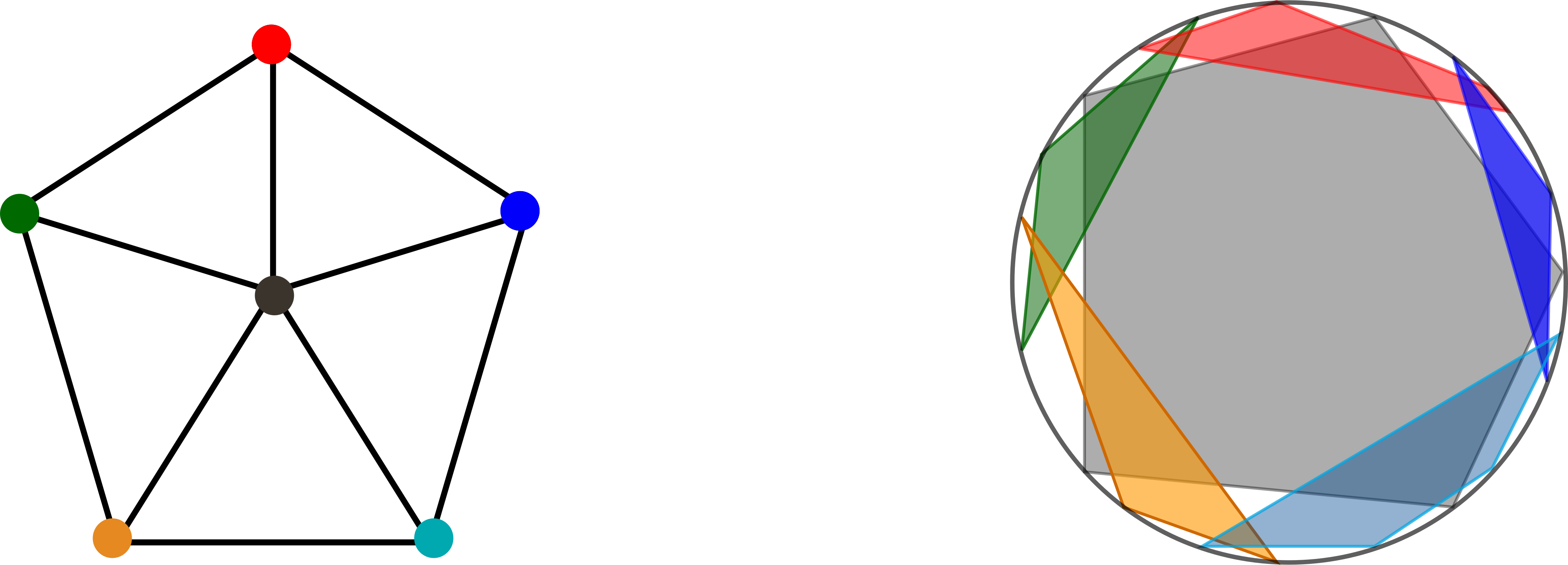}
\caption{$W_5$ and a polygon arrangement that induces it as polygon-circle graph.}
\label{wheel5}
\end{center}
\end{figure}

\begin{proposition}\label{toda_grafica_es_palabra}
Every graph $G=(V,E)$ is general $d$-word-representable for $d \geq m-1$ where $m=|E|$. 
\end{proposition}

\begin{proof}
Assume the edges of $G$ are given by $E=\{e_1,e_2,\dots ,e_m\}$ and $W=W_1W_2\dots W_m$ where every factor $W_i$ is an alternating sequence from alphabet 
$e_i=\{x_i,y_i\}$, $i\in \{1,2,\dots ,m\}$ and length $d+2\geq m+1$. Observe that $W$ $d$-word-represents $G$. Clearly by construction every edge of $G$ is induced by $W$. On the other hand, suppose that $W$ induces an edge $e^*=\{x,y\}\notin E$. Consider an alternating sub-sequence $S \subset W$ with letters $\{x,y\}$ and length at least $d+2$. Note that each letter in $S$ lies in a unique $W_i$; otherwise we have that $e^*\in E$, but this yields a contradiction on the length of $S$. 
\end{proof}

Finding the minimal $d$, $\mathcal{GR}(G)$ such that a graph $G$ is general $d$-word-representable seems an interesting problem on its own.


\subsection{2-word-representable graphs and polygon circle graphs}

\begin{definition}
A graph $G$ is a \emph{polygon-circle} graph if it is the intersection graph of convex polygons inscribed in a circle.
\end{definition}

The $5$-wheel graph $W_5$, for example, belongs to the family of polygon-circle graphs (see Figure \ref{wheel5}).  In \cite{ENRIGHT20193} Enright and Kitaev proved that the two classes of polygon-circle and word-representable graphs do not contain one another. But in the case of general $d$-word-representable graphs we have that the following proposition is true.

\begin{proposition}\label{Polygongraphis2-representable}
A graph $G=(V,E)$ is a polygon-circle graph if and only if $G$ is a general $2$-word-representable graph. 
\end{proposition} 

\begin{proof}
Let $G$ be a general $2$-word-representable graph, and suppose that $W$ is a word, say with length $n$ and alphabet $V$, that general $2$-word-represents $G$. Consider a set of $n$ points in a circle and color them with the labels of the letters of $W$ in the clockwise direction.  Clearly, the convex hulls of the chromatic classes are a polygon arrangement that induces $G$ as a polygon-circle graph. Reciprocally, if $G$ is a polygon-circle graph, and $\tilde{G}$ its realization in the circle, then we may label the vertices of every polygon with a different color. This clockwise coloring sequence starting at any vertex generates a word $W$ that general $2$-word-represents the graph $G$. 
\end{proof}

\begin{observation}\label{rotar}
If a word $W$ general $2$-word-represents a graph $G$, then every word generated by a cyclic permutation of $W$ also general $2$-word-represents  $G$. This clearly does not hold for $d$ odd. Consider the word $W=12121$; this word induces $K_2$ as a general $3$-word-representable graph but its cyclic permutation $W^{\prime}=11212$ induces two disjoint vertices.
\end{observation}

One natural question to ask is what kind of graphs are not general $2$-word-representable. Using a data base from \cite{Combos} we have computationally checked all connected planar graphs with at most nine vertices and have found that the following three planar graphs are not general $2$-word-representable, been the most right not only planar but bipartite (see Figure \ref{no 2-Tverberg}).

\begin{figure}[h!]
\begin{center}
\includegraphics[scale=.7]{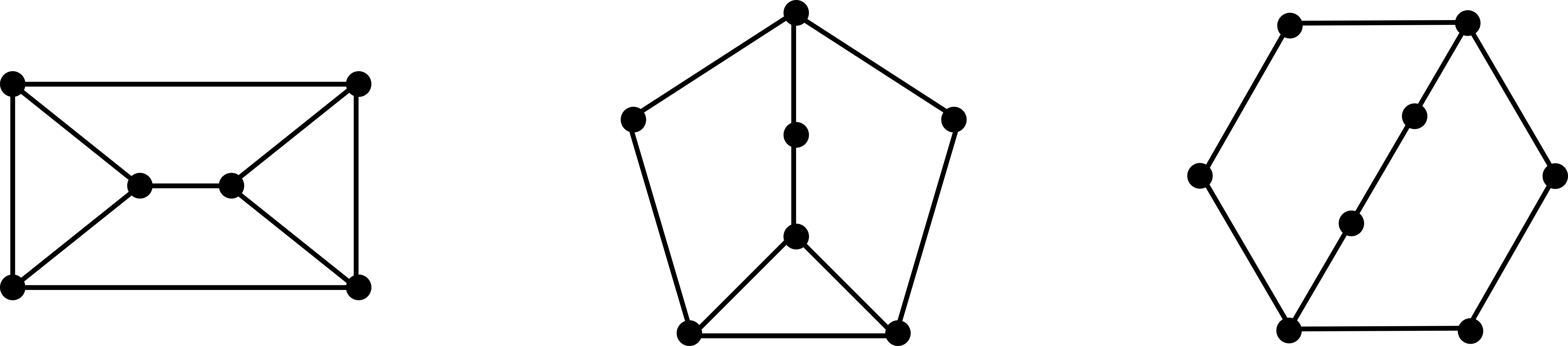}
\caption{ All graphs up to $9$ vertices that are not general $2$-word-representable.}
\label{no 2-Tverberg}
\end{center}
\end{figure}

\section{Nerves over the cyclic polytope and its relation with words}\label{nerves over cyclic polytopes}

The \emph{moment} curve $M_d$ is the curve in $\mathbb{R}^d$ defined by $\{x(t) = (t, t^2,\dots , t^d)\colon t\in \mathbb{R}\}$. Consider a set $T=\{t_1,\dots ,t_r\}\subset\mathbb{R}$ with $r>d+1$ and let $x(T)=\{x(t_i)\colon i = 1,2,\dots , r\}$. Then $\conv(x(T))$ is known as a cyclic $d$-polytope.  As all $r$-point sets in the moment curve are combinatorially equivalent \cite{grunbaum2003convex}, we will refer to any of them as $x(r,d)$ and denote their convex hull by $C(r, d)$.  The following classical theorem gives the combinatorial structure of cyclic polytopes.

\begin{theorem}[Gale's evenness condition \cite{gale1963neighborly}]\label{Gale}
Let $T=\{t_1,\dots ,t_r\}\subset\mathbb{R}$ with $t_1 < t_2 < \cdots < t_r$ and $T_d\subset T$ a subset with $d$ points. Then $C(x(T_d))$ is a facet of $C(x(T))$ if and only if $T\setminus T_d$ are separated by an even number of elements of $T_d$ in the sequence $(t_1,\dots ,t_r)$. 
\end{theorem}
 
In \cite{breen1973primitive} M.\ Breen shows that the set $A\cup B$ is a primitive Radon partition in the cyclic polytope $C(r,d)$ if the following condition holds. 

\begin{theorem}[\cite{breen1973primitive}]\label{Breen_theorem}
Let $A$ and $B$ be two point sets in $V(\vert A  \cup  B \vert ,d)$. Then $\conv(A)\cap \conv(B) \neq \emptyset$ if and only if there exists an  alternating sequence of length $d+2$ along the moment curve.
\end{theorem}

As we will see next, Breen's theorem gives us the key to the relationship between realization of nerves in convex position and  general $k$-word-representable graphs and $k$-word-representable graphs. In particular, Theorem \ref{sin_triangulos_encaje_Md} shows that every triangle-free general $d$-word-representable graph
is partition-induced on a set of points in  $M_d$. Furthermore we are ready to prove Theorem \ref{sin_triangulos_encaje_Md}.


\smallskip
\noindent {\bf{ Proof of theorem \ref{sin_triangulos_encaje_Md}}}
\begin{proof} 
By propositions \ref{d is d+1 } and \ref{toda_grafica_es_palabra} we have that for every $d \geq d_0=\mathcal{GR}(G)$ there exist a word $W$ with  alphabet $V$, such that $W$ general $d$-word-represents $G$, set $m=|W|$.  Let $S$ be any set of $m$ points in general position in dimension $d$. And
let set $S_1$ be a set of $m$ points; $S_1=\{s_1,s_2,\dots ,s_m\}\subset M_{d}$ on the moment curve. Color them by $\mathcal{C}\colon S_1\mapsto V$ defined by $\mathcal{C}(s_i)=W(i)$ and let $P = \{P_1, P_2, \dots , P_m\}$ be the set where $P_i$ is the convex hull of the color class $v_i$. By Theorem \ref{Breen_theorem}, two sets $P_i$ and $P_j$ intersect if and only if colors $v_i$ and $v_j$ are $d$-intersecting in $W$. This implies that $G$ is preserved as the $1$-skeleton of $\mathcal{N}(P)$. Furthermore, since $G$ does not have triangles, there are no intersections of three or more sets of $P$; therefore, there are no faces with a dimension greater than $1$, so we conclude that $\mathcal{N}(P)=G$. Finally we know, than any two set in convex position have the same order type thus, there exist a bijection $\sigma$ that preserves orientation between $S_1$ and $S$. By Lemma 2 of  \cite{TverbergTipeTheorems}, the partition $\mathcal{P} = (P_1, P_2, \dots, P_m)$ of $S_1$ and the corresponding partition of $S$ via $\sigma$, denoted $\sigma{\mathcal{P}} = \{\sigma(P_1), \sigma(P_2), \dots ,\sigma(P_n)\}$,  have the same intersection graph $G$.  
\end{proof}

Observe that since $k$-word-representable graphs are also general $d$-word-representable graphs, using Theorem \ref{sin_triangulos_encaje_Md} and Lemmas \ref{l1}, \ref{l2} and \ref{l4} the following results are true. 

\begin{theorem}\label{bipartiteinduce}
Every bipartite graph $G=(U\cup V,E)$ is partition-induced for a point set $S\subset M_d$ for some positive integer $d$.
\end{theorem}

\begin{theorem}
Every circle graph (intersection graph of a set of chords with vertices on a circle) $G=(V,E)$ is partition-induced for a point set $S\subset M_d$ for $d\geq2$.
\end{theorem}

\section{Extensions and Tverberg-type theorems}\label{Tverberg type theorems and extensions}

In a similar way to \cite{TverbergTipeTheorems}, we will extend the colorings from specific point sets to  more general configurations. Let us define the extension concept formally. 

\begin{definition}\label{defextendable}
Let $S$ be a set of points in $\mathbb{R}^d$ and let $P$ be a color partition of $S$ into color classes $P=S_1,S_2,\dots , S_m$. Then there is a function  $\mathcal{C}\colon S\mapsto \{c_1,c_2,\dots ,c_m\}$, a ``coloring" that yields a specific nerve $\mathcal{N}(P)$. We say that  $\mathcal{C}$ is extendable if for every point set $\overline{S}$ containing $S$, there is a color partition $\overline{P}=\overline{S}_1,\overline{S}_2,\dots \overline{S}_m$ and a coloring $\overline{\mathcal{C}}\colon \overline{S}\mapsto \{c_1,c_2,\dots ,c_m\}$ such that $\overline{C}|_{S}=\mathcal{C}$ and $\mathcal{N}( \overline{P})$ is isomorphic to $\mathcal{N}(P)$.

\end{definition}

In this section we will show that every simplicial complex $\mathcal{K}$ that is partition-induced on a set of points in convex position in the plane is extendable; therefore, it is $2$-Tverberg.  It turns out that in dimension $d\geq 3$, the extension is not that clear (see Figure 3); however, we will show that some subfamilies of graphs are extendable.

\subsection{Extensions in $\mathbb{R}^2$}

We begin by extending colorings of partitions of point sets in convex position contained in $\mathbb{R}^2$, and obtain the following theorem. 

\begin{theorem}\label{2-extension}
Every simplicial complex $\mathcal{K}$ that is partition-induced in any set of points in convex position is $2$-Tverberg. 
\end{theorem}
\begin{proof}
Suppose that a simplicial complex $\mathcal{K}$ is partition-induced in a set of $n$ points in convex position in the plane. In order to show that $\mathcal{K}$ is $2$-Tverberg, we need to show that it is partition-induced in an arbitrary set of points with sufficiently many points.  By the Erd\H{o}s-Szekeres Theorem\cite{Erdos1987}, every point set $\bar{S}$ with sufficiently many points contains a subset ${S}$ of $n$ points in convex position. By hypothesis, there is a partition ${P}$ of $S$ into color classes ${S_1},{S_2},\dots {S_m}$ obtained by a coloring function $\mathcal{C}\colon S\mapsto \{c_1,c_2,\dots ,c_m\}$  such that its nerve $\mathcal{N}({P})=\mathcal{K}$. (Furthermore such a partition yields a polygon circle graph and by Proposition \ref{Polygongraphis2-representable}, it is a $2$-word-representable graph.) It remains to show that we can extend $\mathcal{C}$ to the set $\bar{S}$ with a new coloring $\bar{\mathcal{C}}$ as in Definition \ref{defextendable}.

We  proceed by induction on the number $m$ of different chromatic classes of $\mathcal{C}$.
If there are only two different colors, say $c_1$ and $c_2$,  then two possible scenarios may occur: $\conv(S_1)\cap \conv(S_2)=\emptyset$ or $\conv(S_1)\cap \conv(S_2)\not=\emptyset$. If the first case happens, it is clear that there exists a support line $h$ through one edge of the polygon $\conv(S_1)$ that leaves $\conv(S_1)$ on one side of $h$, say $H^+$,  and $\conv(S_2)$ on the other, $H^-$. Next, we color the points in $\bar{S}\subset H^+$ with color $c_1$ and the points in $\bar{S}\subset H^-$ with color $c_2$.  If $\conv(S_1)\cap \conv(S_2)\not=\emptyset$, then we may color every point in $\bar{S}\setminus S_1$ with color $c_2$.  

Given a color $c_i\in \{c_1,c_2,\dots ,c_m\}$ and a supporting line $h_{ij}$ of $\conv(S_i)$, denote by $H^+_{ij}$ the half-plane that contains $S_i$ (if $S_i$ consists of an interval take either side) and denote by  
 $\mathcal{C}_{ij}^{c}=\{c_{ij_1},c_{ij_2},\dots ,c_{ij_{n_{ij}}}\}$ the set of colors such that the corresponding sets $S_{ij_k}$ are in $H^+_{ij}$ and  
 $\conv(S_i)\cap \conv(S_{ij_k})=\emptyset$ for $k=1,\dots ,n_{ij}$. 
 

Note that if  $|\mathcal{C}_{ij}^{c}|=0$ for some $i$ and some $j$, then we can ignore the points of $S$ having color $c_i$, and by the induction hypothesis there exists an $m-1$ coloring on the remaining points that induce $G$ as a nerve. Now we may re-color all the points in $\bar{S}\setminus S$ that lie in $H_{ij}^+$ with color $c_i$ preserving $G$ as a nerve.

%
%
%

For every color $c_i$, we may suppose that for every supporting line $h_{ij}$ we have that $|\mathcal{C}_{ij}^c|\not=0$. 
Assume that $c_1$ is the color and $h_{11}$ the  supporting line of $conv(S_1)$ such that  $|\mathcal{C}_{11}^c|$ is minimum for all $i=1,\dots ,m$. We now observe that there is a color, say $c_{r}\in \mathcal{C}_{11}^{c}$, and  therefore it is in $H_{11}^+$ and a supporting line $h_{rs}$ such that $|\mathcal{C}_{rs}^{c}|\leq |\mathcal{C}_{11}^{c}|$, contradicting the minimality of $|\mathcal{C}_1^c|$. %
%
\end{proof}

\textbf{Proof of Theorem \ref{mainTheo1} and Theorem \ref{maintheo2circlegraph}}
\smallskip




Although it is possibly true that all general $2$-word-representable graphs are $2$-Tverberg, as it was pointed out in Lemma 2 and Fig. 2 of  \cite{TverbergTipeTheorems}, dealing with simplicial complexes that contain triangles is hard, simply because order types do not register the difference between a $2$-simplex and a triangle. 

By Proposition \ref{Polygongraphis2-representable}, we know that  for every general $2$-word-representable graph $G=(V,E)$, $|V|=m$, there exists an  arrangement of polygons $Q_1,Q_2,\dots ,Q_{m}$ on a circumference that induces $G$ as an intersection graph.  Furthermore,  the $1$-skeleton of every triangle-free graph 
satisfies that it is equal to the intersection graph of its polygon arrangement. Then Theorem \ref{2-extension} yields Theorem \ref{mainTheo1}. 
This implies in particular that $m$-cycles, for $m\geq 3$, are $2$-Tverberg.  

\begin{observation}\label{cord_arragement}
Given a \emph{circle graph} $G=(V,E)$; that is, the intersection graph of a set of chords with vertices on a circle, it is not difficult to see that for a large enough set of points in convex position we can choose $2|V|$ of them and construct a chord arrangement without triple chord intersections such that it induces $G$ as a circular graph.  
\end{observation}
Then by  Theorem \ref{2-extension}, it is clear that the family of \emph{circle graphs} are $2$-Tverberg, which proves Theorem \ref{maintheo2circlegraph}. This implies that families of graphs such as outerplanar graphs, trees and cycles  \cite[pp.~207-210]{sachs1985graphs} are $2$-Tverberg.

\subsection{Extensions in $\mathbb{R}^d$}

As we observed in the previous subsection, Theorem \ref{2-extension} allows us to extend any  given partition  on a sufficiently large set of points $S$ in convex position in the plane to an arbitrary set of points $\bar{S}$ containing $S$, preserving the nerve of the partition in $S$. This theorem uses the fact that a color $c$ and a  line $h$  always exist, leaving all the points of color $c$ on one half-plane together with all possible points such that their color classes intersect with color $c$.

%
%
Unfortunately this is not always the case in higher dimensions. Consider, for example, a set $S=\{x (t_i)\colon 1 \leq i \leq 9 \text{ with } t_1<t_2<\cdots <t_9\}$ and $\mathcal{C}\colon S\mapsto \{r, b, g \}$ such that $\mathcal{C}\{x (t_1), x (t_2), x (t_6) \} = b$, $\mathcal{C} \{x(t_3), x(t_5), x (t_7) \} = r $ and $ \mathcal{C} \{x (t_4), x (t_8), x (t_9) \} = g $ (see Figure \ref{tresciclico}).  

\begin{figure}[h!]
\centering
\includegraphics[scale=.32]{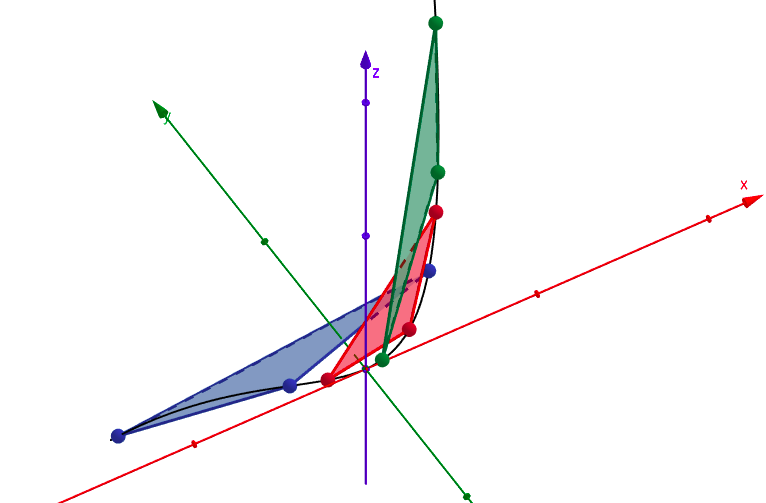}
\caption{Convex hulls of the chromatic classes of $S$.}\label{tresciclico}
\end{figure}

\subsection{Bipartite graphs as general $k$-word-representable graphs}\label{Bipartite graphs}

Although we know that bipartite graphs are $k$-word-representable graphs \cite{WordsandGraphsBook} there is still no specific word description.  In this section, we will prove that bipartite graphs are general $d$-word-representable graphs for some $d$. This implies,  from Theorem \ref{sin_triangulos_encaje_Md}  that bipartite graphs are  a partition induced on a set $S$ with as many points as the length of the 
corresponding word in the moment curve $M_d$. Having the great advantage that such partition  is extendable, will later on allow us to prove Theorem \ref{maintheo3bipartition}.

\begin{theorem}\label{TodaGsepuede}
Every bipartite graph $G=(U\cup V, E)$ is general $d$-word-representable for $d\geq |V|$.
\end{theorem}

\begin{proof}
We start to labeling the vertices of $G$ as $V=\{v_1,v_2,...,v_{d} \}$ and $U=\{u_1,u_2,...,u_{m} \}$ with $d\leq m$. For every pair of vertices $v_i \in V$ and $u_j\in U$, $1\leq i\leq d$ and $1\leq j\leq m$ we define the word $F_i({u_j})$ as an alternating sequence of length $d+2$ with letters $\{v_i,u_j\}$ if $\{v_i,u_j\}\in E$, and as the empty word if $\{v_i,u_j\}\notin E$. 
Next define the word $W_i$ as follows.


\[
  W_{i} = 
  \left \{
    \begin{aligned}
      & F_i({u_{1}})F_i(u_{2})...F_i(u_{m})& \text{ when } i \text{ odd}\\
      & F_i(u_{m})F_i(u_{m-1})...F_i(u_{1})& \text{ when } i \text{ even}
    \end{aligned}
  \right .
\]

Now we claim that $W=W_1W_2....W_{d}$ general $d$-word-represents $G$.

By construction  it is clear that for every edge $e\in E(G)$ there is an alternating sequence of length $d+2$ in $W$ with alphabet in $e$, then every edge of $G$ is induced by $W$. 
Now, we will prove that $W$ does not induce further edges. Consider two different vertices $v_p,v_q\in V$, since these vertices are not adjacent they appear in different factors $W_p$ and $W_q$. Then, note that every alternating sequence with maximum length in $W$ with letters $\{v_p,v_q\}$ has at most length $2$. Therefore $W$ does not induce the edge $\{v_p,v_q\}$. \\
For $\{u_p,u_q\}\notin E$ with $u_p,u_q\in U$, by construction every $W_i$ 
contributes at most with a copy of $u_p$ and at most one of $u_q$ to every alternating sequence with letters $\{u_p,u_q\}$. Furthermore, these elements are ordered in the reverse order in any two consecutive factors $W_i$'s, then every alternating sequence has at most length $n+1$ and $W$ does not induce the edge $\{u_p,u_q\}$.

Finally, given an edge $\{v_p,u_q\}\notin E(G)$ with $v_p\in V$ and $u_q\in U$, by construction there exist a factor $W^*\supseteq W_p$ in $W$ containing all the copies of $v_p$ and only other letters for which there are alternating subsequences of size $d+2 $ with $v_p$. Then, $u_q \notin W^*$ and the maximum length  of an alternating sequence with these two letters in $W$ has a size of at most $3$.
We conclude that $W$ general $d$-word-represents $G$. 
\end{proof}

\subsection{General-bipartite graphs  are \emph{k}-Tverberg complexes}

We begin with the following useful Lemma that follows directly from Gales's evenness condition Theorem \ref{Gale}. 

\begin{lemma}\label{separation}
Every hyperplane $H$ generated by  $d$ points on the moment curve $M_d$,  $H=\langle  \{ p_i \} \rangle := \langle \{p_i =x(l_i)\in M_d\colon l_i \in \mathbb{R} \text{ such that } l_1<l_2< \cdots <l_d\} \rangle$ divides $\mathbb{R}^d$ into two half-hyperplanes one of them, say $H^+$  containing all points lying in $R_i$ for even $i$ and the other $H^-$ containing all points lying in $R_i$ for odd $i$. Where $R_0=\{x(l)\colon l<l_1\}$, $R_i=\{x(l)\colon l_{i-1}<l<l_i\}$ with $1<i\leq d$ and $R_{d+1}=\{x(l)\colon l>l_d\}$.
\end{lemma}

Let $G=(U\cup V,E)$ be a bipartite graph with $V=\{v_1,v_2,...,v_{d}\}$ and $U=\{u_1,u_2,...,u_{m}\}$ with $d\leq m$. Let $W=W_1W_2...W_{d}$ be the word that general $d$-word-represents $G$  constructed as described in the Theorem \ref{TodaGsepuede} with $d\geq|V|$. Consider a point set $S=\{s_i=x(t_i): t_1<t_2<...<t_{|W|} \in \mathbb{R}\}$ of $|W|$ points on the moment curve $M_d$. Then, by Theorem \ref{sin_triangulos_encaje_Md} the coloring $\mathcal{C}_{G}:S\mapsto V$ given by $\mathcal{C}_G(s_i)=W(i)$ induce a partition $P_G$ such that $\mathcal{N}(P_G)=G$. 

\begin{lemma}\label{extendablebiparite}
Given any set of points $\bar{S}$ such that $S\subset \bar{S}\subset \mathbb{R}^d$, $\mathcal{C}_G$ is extendable in such a way that
$\mathcal{N}(\bar{P}_G)=G$. 
\end{lemma}

 \begin{proof} 
 
Let us begin by looking closely to the $W=W_1W_2...W_{d}$ described in Theorem \ref{TodaGsepuede}.  

\small{
$W=\underbrace{F_1({u_{1}})F_1(u_{2})\dots F_1(u_{m})}_{W_1}
       \underbrace{F_2(u_{m})F_2(u_{m-1})\dots F_2(u_{1})}_{W_2} F_3(u_1) \cdots F_{d-1}({u_{1}}) \underbrace{F_{d}({u_{1}})F_{d}(u_{2})...F_{d}(u_{m})}_{W_{d}}
$} if $d$ is odd, or

\small{
$W=\underbrace{F_1({u_{1}})F_1(u_{2})\dots F_1(u_{m})}_{W_1}
      \underbrace{F_2(u_{m})F_2(u_{m-1})\dots F_2(u_{1})}_{W_2} F_3(u_1) \cdots F_{d-1}(u_m) \underbrace{F_{d}(u_{m})\dots F_{d}(u_{2})F_{d}(u_{1})}_{W_{d}}
$}

if $d$ is even. 

Since some of the factors $F_i(u_{1})$ may be empty for some $1\leq i\leq d$. We begin by observing that the factors that contain the letter $u_1$ can be 
grouped in most $\lfloor\frac{d}{2}\rfloor +1 $  subwords of length one or two.
This observation will allow us to insert $d$ ``separators"  $\bf{p_1,p_2,\dots ,p_d}$ that isolate these $\lfloor\frac{d}{2}\rfloor +1 $ subwords containing letter $u_1$, as follows:

$\tilde{W}=\underbrace{F_1({u_{1}}){\bf{p_1}}F_1(u_{2})\dots F_1(u_{m})}_{W_1}
       \underbrace{F_2(u_{m})F_2(u_{m-1})\dots {\bf{p_2}} F_2(u_{1})}_{W_2}  F_3(u_1){\bf{p_3}}\dots$ 
\begin{flushright}
$\dots \hspace{.5cm} {\bf{p_{d-1}}}F_{d-1}({u_{1}}) \underbrace{F_{d}({u_{1}}) {\bf{p_{d}}} F_{d}(u_{2})...F_{d}(u_{m})}_{W_{d}}
$ if $d$ is odd, or 
\end{flushright}       
 
$ \tilde{W}=\underbrace{F_1({u_{1}}){\bf{p_1}}F_1(u_{2})\dots F_1(u_{m})}_{W_1}
       \underbrace{F_2(u_{m})F_2(u_{m-1})\dots {\bf{p_2}} F_2(u_{1})}_{W_2}  F_3(u_1){\bf{p_3}}\dots$
\begin{flushright}
$\dots \hspace{1.7cm}  \underbrace{F_{d}(u_{m})\dots F_{d}(u_{2}){\bf{p_{d}}}F_{d}(u_{1})}_{W_{d}}
$ if $d$ is even. 
\end{flushright}

Now, these separators will not be part of $S$ and therefore will not be colored by $\mathcal{C}_G$. 
The aim of these separators is only to choose a point in $M_d\setminus S$  in the corresponding interval between the two original points in $S$. 
That is, if  $W=\dots W(k){\bf{p_r}}W(k+1)\dots $ then we may chose any $t$ such that $t_k<t<t_{k+1}$ and  $p_r:=x(t)$ for $1\leq r\leq d$. 
This will allow us to choose $d$ points $\{p_1,p_2,\dots ,p_d\}$ at the moment curve $M_d$. 
By Lemma \ref{separation}, there exists a hyperplane $H_1=\langle  p_i \rangle$, containing no points of $\overline {S}$, that divides $\mathbb{R}^d$ in two half-planes $H_1^+$ and $H_1^-$. This is done so that the points of $S$, that are related with a non-empty factor of the form $F_i(u_1)$ colored with colors $u_1$ or $v_i$, lie on one of the half-planes, say $H_1^+$, and every other point in $S$ lies in $H_1^-$. 
Similarly, for color $u_2$ we can define a new hyperplane $H_2$ and a convex region $R_2= (\mathbb{R}^d\setminus H_1^+) \cap H_2^+$ such that every point related with the factors of the form $F_i(u_2)$ lie in the interior of this convex region $R_2$.  Recursively, for $3\leq j \leq m-1$ we can define hyperplanes $H_j$ and convex regions $R_j= (\mathbb{R}^d\setminus \cup_{k=1}^j R_{k-1}) \cap H_j^+$ such that every point related with the factors of the form $F_i(u_j)$  lies in the interior of these convex regions. 
Finally we define the region $R_{m}= \mathbb{R}^d \setminus R_{m-1}^+$.\\
Observe that for every point $s\in S$ lying in the interior of the region $R_j$, we have one of two cases: $\mathcal{C}(s)=u_i$ or the convex hull of the points with color $\mathcal{C}(s)$ intersects the convex hull of the points with color $u_j$. 
By construction of the regions $R_j$, every pair of adjacent points of $S$ with color $v\in V$ lie in the same region or in adjacent regions. Then, if $v\notin R_j$ for some $j$, the convex hull of the points with color $v$ does not intersect the convex hull of the points with color $u_j$. 
Coloring the points of $(\bar{S}\setminus S) \cap R_j$ with the color $u_j$ yields that the resulting coloring extension preserves $G$ as a nerve.    
\end{proof}
 
\noindent{\bf Proof of Theorem \ref{maintheo3bipartition} }

\begin{proof} Consider a bipartite graph $G=(U\cup V,E)$ and let  
 $W$ a word of length $|W|$ that general $d$-word represents $G$ for some integer $d$, constructed as described in Theorem \ref{TodaGsepuede}. 
By the multi-dimensional version of Erd\H{o}s-Szekeres theorem (due to Grunbaum\cite{Gbook} and Cordovil and Duchet \cite{CordovilDuchet}, we know that 
every large enough set of points in $\bar{S}\subset \mathbb{R}^d$ contains a set of $|W|$ points $S$ of some curve $\gamma $ that is combinatorially equivalent to a set $S_1$ on the moment curve $M_d$. Then by Lemma 2  of  \cite{TverbergTipeTheorems}   there is a bijection $\sigma$ from $S_1$ to $S$ that 
preserves the orientation of any $(d+1)$-tuple in $S$. Then any partition $\mathcal{P} = (P_1, P_2, \dots, P_n)$ of $S_1$ and the 
corresponding partition of $S$ via $\sigma$, denoted $\sigma{\mathcal{P}} = \{\sigma(P_1), \sigma(P_2), \dots ,\sigma(P_n)\}$,  have 
the same intersection graph $\mathcal{N}^1(\mathcal{P})$.  Since bipartite graphs are triangle-free simplicial complexes, their nerve complexes equal to their $1$-skeleton.
This implies that the partition induced by the word $W$ in  Theorem \ref{sin_triangulos_encaje_Md} is also  partition induced in $S$. Then finally by 
Lemma \ref{extendablebiparite}  $G$ is $d$-Tverberg. 
\end{proof}

\section{Conclusions}

In this paper we have obtained new Ramsey-type Tverberg theorems by showing that some infinite families of graphs (such as cycle graphs and bipartite graphs) are always induced as the nerve of some partition of any sufficiently large set of points in $\mathbb{R}^d$, and we relate $k$-word-representable and general $k$-word-representable graphs to realizations of nerves in convex position in $\mathbb{R}^d$. We believe that it is possible to generalize Theorem \ref{mainTheo1} to every general $2$-word-representable graph without the triangle-free restriction, but at this time it seems difficult to prove.

We observe that if a graph $G$ is general $d$-word-representable and extendable for some integer $d$, then $G$ is $d$-Tverberg. Furthermore, if a graph $G$ is not general $d$-word-representable for a fixed $d$, then $G$ is not partition-induced in $\mathbb{R}^d$. Thus it would be interesting to study how to determine the minimum $d$ for which a certain given graph $G$ is general $d$-Tverberg. This would provide us with a tool to see how to decide if a given graph $G$ is not $(d-1)$-Tverberg. 
We conjecture that the following bipartite graph (see Figure \ref{conjetura}) is not $3$-Tverberg, but even with the assumption that each vertex contributes only  $3$ copies to the possible word, there are $\sum_{i=1} ^ {10}{{3i}\choose{3}}$ words to be checked, which is hard to compute. Then clearly another approach for future work is needed. 

\begin{figure}[h!]
\centering
\includegraphics[scale=.8]{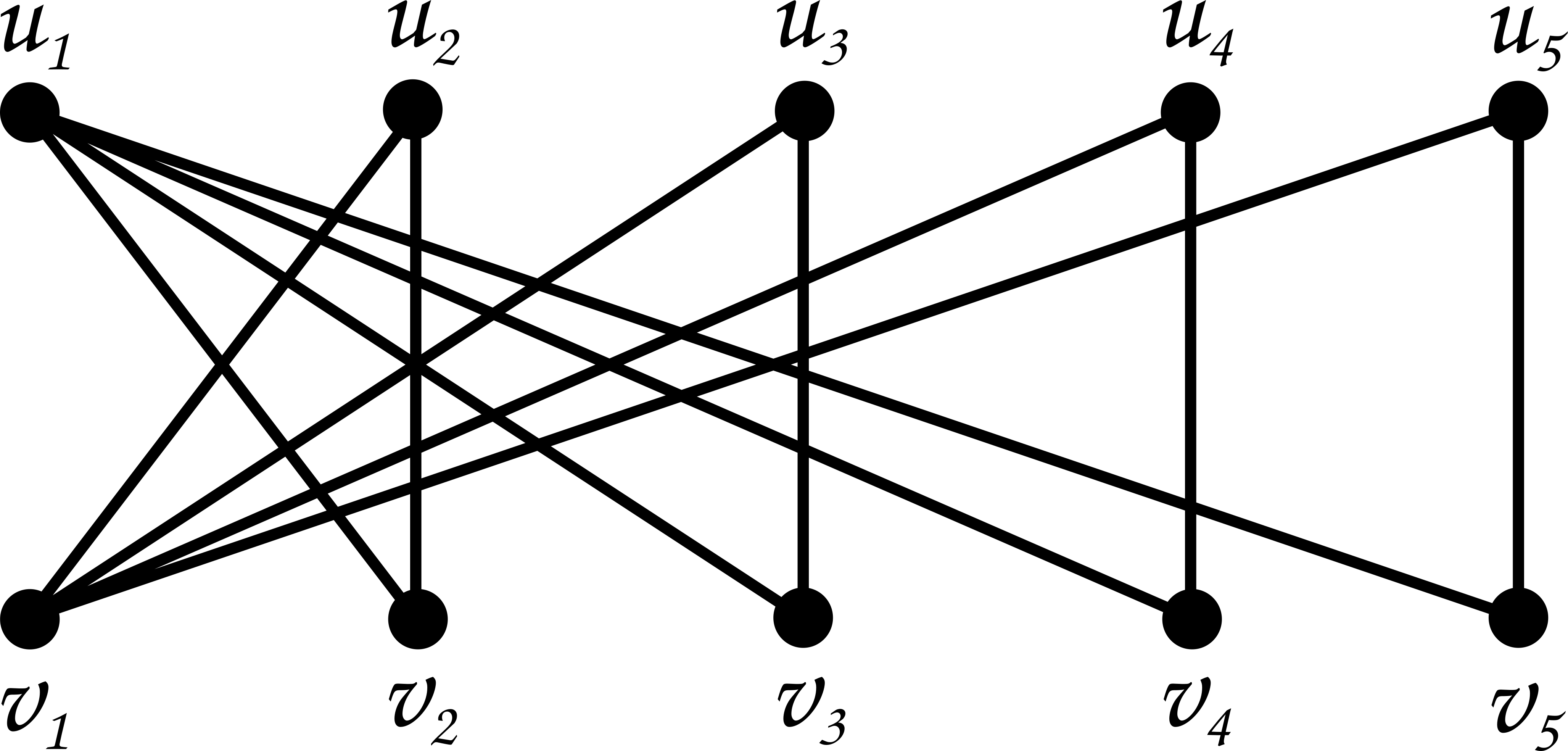}
\caption{Graph that may not be 3-Tverberg.}\label{conjetura}
\end{figure}

\section*{Acknowledgements}
The authors will like to thank the support from grants PAPIIT 104915 and CONACyT 166306.

\bibliographystyle{alpha}
\bibliography{references_bipartite}
\end{document}